\documentclass{amsart}
\pdfoutput=1

\usepackage{amsmath,amssymb,latexsym}
\usepackage{mathtools}
\usepackage[utf8]{inputenc}
\usepackage{units}
\usepackage{enumitem}
\usepackage{array}

\usepackage{placeins}
\usepackage[obeyspaces,hyphens,spaces]{url}

\usepackage{graphicx,color}
\graphicspath{{img/}}

\usepackage{tikz}
\usetikzlibrary{matrix,arrows,calc}

\usepackage{csquotes}
\usepackage[style=alphabetic,backend=bibtex]{biblatex}
\bibliography{sources}

\usepackage{amsthm}
\usepackage{hyperref}

\makeatletter
\ifcsname phantomsection\endcsname
    \newcommand*{\@gobblenexttocentry}[9]{}
\else
    \newcommand*{\@gobblenexttocentry}[4]{}
\fi
\newcommand*{\addsubsection}{%
    \addtocontents{toc}{\protect\@gobblenexttocentry}%
    \subsection*}
\makeatother

\usepackage{aliascnt}

\newcommand{\mynewtheorem}[4]{
  \if\relax\detokenize{#3}\relax 
    \if\relax\detokenize{#4}\relax 
      \newtheorem{#1}{#2}
    \else
      \newtheorem{#1}{#2}[#4]
    \fi
  \else
    \newaliascnt{#1}{#3}
    \newtheorem{#1}[#1]{#2}
    \aliascntresetthe{#1}
  \fi
  \expandafter\def\csname #1autorefname\endcsname{#2}
}

\def\equationautorefname~#1\null{(#1)\null}

\mynewtheorem{theorem}{Theorem}{}{section}
\mynewtheorem{lemma}{Lemma}{theorem}{}
\mynewtheorem{rem}{Remark}{lemma}{}
\mynewtheorem{prop}{Proposition}{lemma}{}
\mynewtheorem{cor}{Corollary}{lemma}{}
\mynewtheorem{dfn}{Definition}{lemma}{}

\def\defbb#1{\expandafter\def\csname#1\endcsname{\mathbb{#1}}}
\def\defcal#1{\expandafter\def\csname#1\endcsname{\mathcal{#1}}}
\def\deffrak#1{\expandafter\def\csname#1\endcsname{\mathfrak{#1}}}
\def\defop#1{\expandafter\def\csname#1\endcsname{\operatorname{#1}}}

\makeatletter
\def\defcals#1{\@defcals#1\@nil}
\def\@defcals#1{\ifx#1\@nil\else\defcal{#1}\expandafter\@defcals\fi}
\def\deffraks#1{\@deffraks#1\@nil}
\def\@deffraks#1{\ifx#1\@nil\else\deffrak{#1}\expandafter\@deffraks\fi}
\def\defbbs#1{\@defbbs#1\@nil}
\def\@defbbs#1{\ifx#1\@nil\else\defbb{#1}\expandafter\@defbbs\fi}
\def\defops#1{\@defops#1,\@nil}
\def\@defops#1,#2\@nil{\if\relax#1\relax\else\defop{#1}\fi\if\relax#2\relax\else\expandafter\@defops#2\@nil\fi}
\makeatother

\defbbs{CPZRQD}

\defcals{MAEOXY}
\newcommand{\CC}{\mathcal{C}}

\deffraks{b}

\defops{End,Aut,Aff,Jac,SL,GL,PSL,SO,Stab,PStab,Re,Im,div,Gal,mod}

\makeatletter%
\protected\def\W{\@ifnextchar[{\W@arg}{W_D}}
\def\W@arg[#1]{%
  \def\@D{D}%
  \def\c@W##1,##2,##3\@nil{\def\c@D{##1}\def\n@D{##2}}\c@W#1,,\@nil%
  \def\@W##1=##2=##3\@nil{\def\@DD{##1}\def\@DDD{##2}}\expandafter\@W\c@D==\@nil%
  \ifx\@D\@DD W_{\@DDD}%
    \expandafter\ifx\n@D\relax\relax\else(\n@D)\fi%
  \else W_D(#1)\fi%
}
\protected\def\parexp#1{\@ifnextchar^{(#1)}{#1}}
\protected\def\parind#1{\@ifnextchar_{(#1)}{#1}}
\makeatother%

\renewcommand{\i}{\mathrm{i}}
\renewcommand{\d}[1]{\mathrm{d}#1}
\renewcommand{\H}
{\mathrm{H}}

\def\ep{\varepsilon}

\def\abs#1{\lvert#1\rvert}

\begin{document}
\title[Galois Action and Spin in Genus 3]{The Galois Action and a Spin Invariant for Prym-Teichmüller Curves in Genus 3}
\author{Jonathan Zachhuber}
\thanks{The author was partially supported by ERC-StG 257137.}
\address{FB 12 -- Institut für Mathematik\\Johann Wolfgang Goethe-Universität\\Robert-Mayer-Str. 6--8\\D-60325 Frankfurt am Main}
\email{zachhuber@math.uni-frankfurt.de}

\begin{abstract}
Given a Prym-Teichmüller curve in $\M_3$, this note provides an invariant that sorts the cusp prototypes of Lanneau and Nguyen by component. This 
can be seen as an analogue of McMullen's genus $2$ spin invariant, although the source of this invariant is different.
Moreover, we describe the Galois action on the cusps of these Teichmüller curves, extending the results of Bouw and Möller in genus $2$. We use this to show that the components of the genus $3$ Prym-Teichmüller curves are homeomorphic.
\end{abstract}
\maketitle

\tableofcontents

\section{Introduction}

A \emph{Teichmüller curve} is a curve inside the moduli space $\M_g$ of smooth projective genus $g$ curves that is totally geodesic for the Teichmüller metric. Every Teichmüller curve arises as the projection of the $\GL_2^+(\R)$ orbit of a flat surface (see \autoref{cusps} and the references therein for background and definitions). Only a few infinite families of primitive Teichmüller curves are known. McMullen constructed several families in low genera, among them, for every discriminant $D$, the \emph{Prym-Teichmüller} or \emph{Prym-Weierstraß} curves $\W$ in genus $3$ \cite{mcmullenprym}.

This family is fairly well understood. In particular, Möller calculated the Euler characteristic \cite{moellerprym}, Lanneau and Nguyen enumerated the cusps and connected components \cite{lanneaunguyen}, and the number and type of orbifold points are determined in \cite{TTZ}. The aim of this note is to complete the classification of the topological components by showing that the connected components of $\W$ are always homeomorphic.

To be more precise, in \cite{lanneaunguyen}, Lanneau and Nguyen show that $\W$ has at most two components for any $D$ and has two components if and only if $D\equiv 1\mod 8$.

\begin{theorem}\label{comphomeo}
Let $D\equiv 1\mod 8$, which is not a square. Then the two components of $\W$ are homeomorphic.
\end{theorem}

A similar result was obtained by Bouw and Möller 
\cite{BMFuchs} for Teichmüller curves in genus $2$. Note that a Teichmüller curve is always defined over a number field but is never compact. Both approaches rely on determining the stable curves associated to the cusps of the Teichmüller curve and describing explicitly the Galois action on these cusps. At this point, it is crucial that we are able to determine of a pair of cusps if they lie on the same component or not. In genus $2$, Bouw and Möller could use McMullen's spin invariant \cite{mcmullenspin} to achieve this.


However, while Lanneau and Nguyen list prototypes corresponding to the cusps of Prym-Teichmüller curves \cite{lanneaunguyen}, they do not provide an effective analogue of the spin invariant. Here we give such an invariant, which is, moreover, easy to compute.


\begin{theorem}\label{spin}
Let $D\equiv 1\mod 8$, which is not a square. Given a cusp prototype $[w,h,t,e,\ep]$ (see \autoref{cusps}), the associated cusp of $\W$ lies on the component $\W^i$ if and only if
\[2i\equiv e+\ep\mod 4,\]
for $i=1,2$.
\end{theorem}

In \autoref{componentsandspin}, we prove \autoref{spin} essentially using topological arguments. More precisely, we analyse the intersection pairing on a certain intrinsic subspace of homology with $\Z/2\Z$ coefficients. This is similar to the approach of \cite{mcmullenspin} where the Arf invariant of a quadratic form that was associated to the flat structure was analysed on such a subspace, but the nature of these subspaces is different (cf. \cite[Remark 2.9]{lanneaunguyen}). Note also that in genus $3$ the two components lie on disjoint Hilbert Modular Surfaces (cf. \cite[Proposition 4.6]{moellerprym}) and that the $(1,2)$-polarisation of the Prym variety plays a special role in this case, essentially yielding a much more compact formula (cf. \cite[Theorem 5.3]{mcmullenspin}).

In \autoref{galoisaction}, we proceed to give an explicit description of the Galois action on Lanneau and Nguyen's cusp prototypes (\autoref{galoiscusps}) and combine this with \autoref{spin} to show that Galois-conjugate cusps always lie on different components of $\W$, thus proving \autoref{comphomeo}.

\addsubsection{Acknowledgements}
I am very grateful to my advisor, Martin Möller, for many helpful discussions and comments. I also thank the anonymous referee for many valuable suggestions, in particular regarding \autoref{3cylinders}. I thank \cite{PARI2} for computational assistance.

\section{Cusp Prototypes}
\label{cusps}

A \emph{flat surface} is a pair $(X,\omega)$ where $X$ is a compact Riemann surface of genus~$g$ and $\omega\in\H^0(X,\omega_X)$ is a holomorphic $1$-form on $X$. Note that $X$ obtains a flat structure away from the zeros of $\omega$ via integrating $\omega$ and affine shearing of this flat structure gives an action of $\GL_2^+(\R)$.
A \emph{Teichmüller curve} is a $\GL_2^+(\R)$ orbit of a flat surface that projects to an algebraic curve inside the moduli space $\M_g$. See e.g. \cite{parkcity} for background on Teichmüller curves and flat surfaces. Not many families of primitive Teichmüller curves are known; McMullen constructed families in low genera by requiring a factor of the Jacobian of $X$ to admit real multiplication, the (Prym-)Weierstraß curves. We briefly review the construction in genus~$3$, the case with which we are concerned.


\addsubsection{Real Multiplication}
Let $D\equiv 0,1\mod 4$ be a (positive) non-square discriminant and denote by $\O_D$ the corresponding order in the real quadratic number field $\Q(\sqrt{D})$. Let $X$ be a genus $3$ curve and $\rho$ an involution with $X/\rho$ of genus $1$. Then we say that $(X,\rho)$ \emph{admits real multiplication by $\O_D$} if there exists an injective ring homomorphism $\iota\colon\O_D\rightarrow\End\H_1(X,\Z)^-$, such that
\begin{itemize}
\item every endomorphism $\iota(s)$ is self-adjoint with respect to the intersection pairing on $\H_1$, and
\item $\iota$ cannot be extended to any $\O_{D'}\supset\O_D$.
\end{itemize}
In other words, the $\rho$-anti-invariant part $\H_1(X,\Z)^-$ of the homology admits a symplectic $\O_D$-module structure and $\O_D$ is maximal in this respect.

Note that here and in the entire paper, we explicitly exclude the case that $D=d^2$ is a square.


\addsubsection{Prym-Weierstraß Curves}
Denote by $\W$ the space of genus $3$ flat surfaces $(X,\omega,\rho,\iota)$ with an involution $\rho$ that admit real multiplication $\iota$ as above and where additionally $\omega$ has a single ($4$-fold) zero, is $\rho$-anti-invariant, and is an eigenform for the induced action of $\O_D$ on $\H^0(X,\omega_X)$. McMullen \cite{mcmullenprym} showed that $\W$ is a union of Teichmüller curves, the \emph{genus $3$ Prym-Weierstraß} or \emph{Prym-Teichmüller curves of discriminant $D$}. Prym-Weierstraß curves have been studied intensely, see e.g. \cite{mcmullenprym}, \cite{moellerprym}, \cite{lanneaunguyen} and \cite{TTZ}. Note, in particular, that $\W$ is empty for $D\equiv 5\mod 8$.

Again, we note that we explicitly exclude the case that $D=d^2$ is a square, see \cite[Appendix B]{lanneaunguyen} for some results in this case.


\addsubsection{Cusps} Recall that a Teichmüller curve $\CC$ is never compact. We describe the cusps first in the terminology of flat surfaces. Let $(X,\omega)$ be a flat surface generating $\CC$ and consider a direction $v\in\P^1(\R)$. Recall that  a geodesic segment is said to be a \emph{saddle connection} if its endpoints are (not necessarily distinct) zeros of $\omega$ and the direction $v$ is said to be \emph{periodic} if all geodesics in direction $v$ are either closed or saddle connections. We say that a \emph{cylinder} is a maximal union of homotopic geodesics on $(X,\omega)$ and any closed geodesic inside a cylinder is a \emph{core curve}. The length of a core curve is the \emph{width} of the cylinder. A cylinder is called \emph{simple} if each boundary consists of a simple saddle connection. The cusps of $\CC$ are in one-to-one correspondence with the parabolic \emph{cylinder decompositions} on $(X,\omega)$, see e.g. \cite[\S 4]{mcmullenspin} or \cite[\S 5.4]{parkcity}.


\addsubsection{Prototypes} To describe the cusps of $\W$, Lanneau and Nguyen introduce prototypes that encode the cylinder decompositions \cite[\S 3,4 and C]{lanneaunguyen}. We briefly summarise the results we need.

The following result is a slight refinement of \cite[Proposition 3.2]{lanneaunguyen}.
\begin{lemma}\label{3cylinders}
Given $D$ non-square and a point $(X,\omega)$ on $\W$, any periodic direction decomposes $(X,\omega)$ into three cylinders. 
\end{lemma}

\begin{proof}
By \cite[Proposition 3.2]{lanneaunguyen}, any periodic direction decomposes $(X,\omega)$ into either three cylinders, or two cylinders that are permuted by the Prym involution or one cylinder (that is fixed by the Prym involution). Obviously, in the last two cases, the ratio of cylinder circumferences is $1$. However, \cite[Theorem 1.9]{wrightcylinders} asserts that adjoining the ratio of cylinder circumferences to $\Q$ gives the trace field of $(X,\omega)$, which is $\Q(\sqrt{D})$ (cf. \cite[Corollary 3.6]{mcmullenprym}), a contradiction.
\end{proof}

\begin{rem}
\autoref{3cylinders} can be seen as a converse to \cite[Corollary 3.4]{lanneaunguyen}.
\end{rem}

Following \cite{lanneaunguyen}, after rescaling, applying Dehn-twists, and normalising so that the horizontal direction is periodic, this decomposition may be encoded in a combinatorial prototype
\[P_D=[w,h,t,e,\ep]\in\Z^5\]
subject to the following conditions:
\[\begin{cases}
D=e^2+8wh,\ \ep=\pm1,\ w,h>0,\\
w>\frac{\lambda}{2},\ 0\leq t<\gcd(w,h),\ \gcd(w,h,t,e)=1,
\end{cases}\]
where we set
\begin{equation}\label{lambda}
\lambda\coloneqq\lambda_P\coloneqq\frac{e+\sqrt{D}}{2}.
\end{equation}
Moreover, if $\ep=1$, the stronger condition $w>\lambda$ is required.

\begin{figure}
\centering
\includegraphics[width=\textwidth]{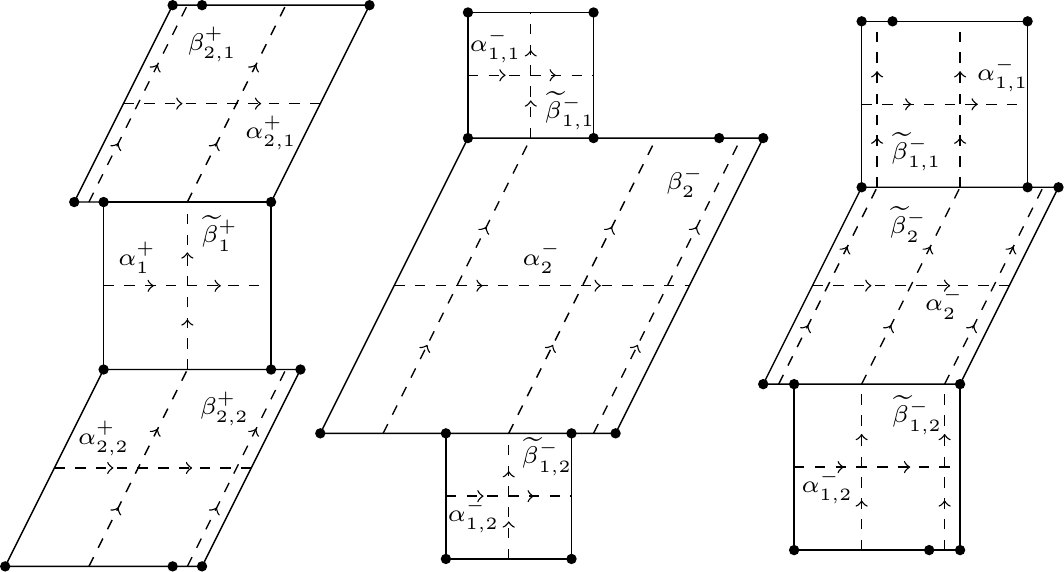}
\caption{Prototypes of geometric type $A+$, $A-$ and $B$. Observe that all $\alpha_i$ are drawn in a horizontal direction, the $\beta_i$ are drawn vertical. We set $\alpha_i=\alpha_{i,1}+\alpha_{i,2}$ and $\beta_i=\beta_{i,1}+\beta_{i,2}$ when appropriate, and furthermore, for the $A+$ prototype, $\beta^+_1=\smash{\widetilde{\beta}}^+_1-\beta^+_2$, for the $A-$ prototype, $\beta^+_{1,i}=\smash{\widetilde{\beta}}^-_{1,i}-\beta^-_2$, and, for the $B$ prototype, $\beta^-_2=\smash{\widetilde{\beta}}^-_{1,1}+\smash{\widetilde{\beta}}^-_{1,2}-\smash{\widetilde{\beta}}^-_2$ and $\beta^-_{1,i}=\smash{\widetilde{\beta}}^-_{1,i}-\beta^-_2$. Thus the $\alpha_i$ and $\beta_i$ give symplectic bases whose periods describe the cylinder heights and widths.}
\label{fig:prototypes}
\end{figure}

Conversely, given a combinatorial prototype, we obtain a three-cylinder decomposition into one of the following three geometric types (see \autoref{fig:prototypes}):
\begin{itemize}
\item $A+$: If $\ep=1$ and $\lambda<w$, we obtain a cylinder decomposition with a single (short) simple cylinder of width and height $\lambda$ and two cylinders of width $w$, height $h$ and twist $t$.
\item $A-$: If $\ep=-1$ and $\lambda<w$, we obtain a cylinder decomposition with two (short) simple cylinders of width and height $\nicefrac{\lambda}{2}$ and a third cylinder of width $w$, height $h$ and twist $t$.
\item $B$: If $\ep=-1$ and $\nicefrac{\lambda}{2}<w<\lambda$, we obtain a cylinder decomposition with no simple cylinders but again two short cylinders of width and height $\nicefrac{\lambda}{2}$ and a third cylinder of width $w$, height $h$ and twist $t$.
\end{itemize}
Each geometric prototype corresponds to exactly one cusp of $\W$.


\section{Components and Spin}
\label{componentsandspin}

In analogy to the situtation in genus $2$, Lanneau and Nguyen showed that, for any discriminant $D$, the locus $\W$ has at most two components \cite[Theorem 2.8, 2.10]{lanneaunguyen}. More precisely, $\W$ has two components if and only if $D\equiv 1\mod 8$. In the following, we denote these components by $\W^1$ and $\W^2$. 

The aim of this section is to provide an analogue of McMullen's spin invariant in genus $2$ \cite{mcmullenspin}, i.e. an invariant that determines if a cusp prototype is associated to a cusp on $\W^1$ or $\W^2$.
%
%


To each geometric prototype $P_D=[w,h,t,e,\ep]$, Lanneau and Nguyen associate a basis $\b=\b(P_D)$ of $\H_1(X,\Z)^-$ \enquote{spanning cylinders}, cf. \cite[\S 4]{lanneaunguyen}. We will see that, in fact, the behaviour of the basis will depend only on $\ep$, i.e. geometric type $A-$ and $B$ will not be distinguished. Hence, we denote the bases by 
\[\b^\ep=(\alpha^\ep_1,\alpha^\ep_2,\beta^\ep_1,\beta^\ep_2),\]
where $\alpha_i$ and $\beta_i$ are as in \autoref{fig:prototypes}. In particular, the periods (with respect to $\omega$) are
\begin{equation}\label{periods+}
\int_{\alpha^+_1}\omega=\lambda,\quad\int_{\alpha^+_2}\omega=2w,\quad\int_{\beta^+_1}\omega=\i\lambda,\quad\int_{\beta^+_2}\omega=2t+2\i h
\end{equation}
if $P_D$ is of geometric type $A+$ (i.e. $\ep=1$) and
\begin{equation}\label{periods-}
\int_{\alpha^-_1}\omega=\lambda,\quad\int_{\alpha^-_2}\omega=w,\quad\int_{\beta^-_1}\omega=\i\lambda,\quad\int_{\beta^-_2}\omega=t+\i h
\end{equation}
if $P_D$ is of geometric type $A-$ or $B$ (i.e. $\ep=-1$).

Moreover, the intersection form on $\H_1(X,\Z)^-$ is of type $(1,2)$. Clearly, it is described by the matrices
\begin{equation}\label{intersectionmatrix}
\langle\cdot,\cdot\rangle_{\b^+}=\begin{pmatrix}\phantom{-}0&\phantom{-}0&1&0\\\phantom{-}0&\phantom{-}0&0&2\\-1&\phantom{-}0&0&0\\\phantom{-}0&-2&0&0\end{pmatrix}
\quad\text{and}\quad
\langle\cdot,\cdot\rangle_{\b^-}=\begin{pmatrix}\phantom{-}0&\phantom{-}0&2&0\\\phantom{-}0&\phantom{-}0&0&1\\-2&\phantom{-}0&0&0\\\phantom{-}0&-1&0&0\end{pmatrix}.
\end{equation}
In particular, $\langle\alpha_i,\alpha_j\rangle=\langle\beta_i,\beta_j\rangle=0$ for any $i,j$ and $\langle\alpha_i,\beta_j\rangle$ is nonzero iff $i=j$.


Recall that, for $D\equiv 1\mod 4$, the quadratic order is $\O_D=\Z\oplus T\Z$, where
\[T=\frac{1+\sqrt{D}}{2}.\]
As $(X,\omega)\in\W$ admits real multiplication $\iota$, $\H_1(X,\Z)^-$ is an $\O_D$-module. In particular, for odd $D$, we may view $T$ as an endomorphism $\iota(T)$ on $\H_1(X,\Z)^-$.
We now describe this endomorphism on the cusp prototypes. Note that this calculation essentially appears already in \cite[\S 4]{lanneaunguyen}, but due to differences in notation and for the convenience of the reader, we briefly restate the result.
\begin{lemma}
Let $D$ be an odd discriminant.
Given a prototype $P_D=[w,h,t,e,\ep]$ associated to a flat surface $(X,\omega)$ the endomorphism $\iota(T)$ acts on $\H_1(X,\Z)^-$ in the basis $\b(P_D)=\b^\ep$ by $\iota(T)_{P_D}=\iota(T)^\ep$, where
\begin{equation*}
\iota(T)^+=\begin{pmatrix}\frac{e+1}{2}&2w&0&2t\\h&-\frac{e-1}{2}&-t&0\\0&0&\frac{e+1}{2}&2h\\0&0&w&-\frac{e-1}{2}\end{pmatrix}
\text{and }
\iota(T)^-=\begin{pmatrix}\frac{e+1}{2}&w&0&t\\2h&-\frac{e-1}{2}&-2t&0\\0&0&\frac{e+1}{2}&h\\0&0&2w&-\frac{e-1}{2}\end{pmatrix}.
\end{equation*}
Note that $e$ is odd iff $D$ is odd.
\end{lemma}

\begin{proof}
Note first that $T=\lambda-\frac{e-1}{2}$ (cf. \eqref{lambda}) and that any $\gamma\in\H_1(X,\Z)^-$ satisfies
\[\int_{\iota(T)\cdot\gamma}\omega=\int_\gamma\iota(T)^\vee\omega=T\cdot\int_{\gamma}\omega,\]
as $\omega$ is an eigenform.
Now, using the periods of $\b^\pm$ in \eqref{periods+} and \eqref{periods-}, as well as the identities $\lambda^2=e\lambda+2wh$ and $T\cdot\lambda=2wh+\lambda\frac{e+1}{2}$, the representations $\iota(T)^\pm$ are obtained by a straight-forward calculation.
\end{proof}

We are now in a position to describe the restriction of the intersection pairing $\langle\cdot,\cdot\rangle$ to the image of the endomorphism $\iota(T)$ in $\H_1(X,\Z/2\Z)^-$ (for $D$ odd). To ease notation, we will no longer distinguish $T$ and $\iota(T)$, as no confusion can arise.

\begin{prop}\label{spininvariant}
Let $D$ be an odd discriminant and $T$ the endomorphism from above. Let $(X,\omega)$ be the geometric prototype associated to the cusp prototype $P_D=[w,h,t,e,\ep]$. Then
\[\langle\cdot,\cdot\rangle|_{\Im T}\equiv 0\mod 2\iff e+\ep\equiv 0\mod 4,\]
where $\langle\cdot,\cdot\rangle|_{\Im T}$ is the restriction of the intersection pairing on $\H_1(X,\Z)^-$ to the image of $T$.
\end{prop}

\begin{proof}
We begin by observing that, as $T$ is self-adjoint by the condition on real multiplication, we have $\langle T\gamma,T\delta\rangle=\langle T^2\gamma,\delta\rangle$ for any $\gamma,\delta\in\H_1(X,\Z)^-$. Moreover, by \eqref{intersectionmatrix}, any two elements $b_1,b_2\in\b^\ep$ satisfy
\[\langle b_1,b_2\rangle\not\equiv 0\mod 2\iff \{b_1,b_2\}=\begin{cases}\{\alpha^+_1,\beta^+_1\},&\text{if $\ep=1$,}\\\{\alpha^-_2,\beta^-_2\},&\text{if $\ep=-1$.}\end{cases}\]
Therefore, by checking mod $2$ the $1,1$ entry of $(T^+)^2$ and the $2,2$ entry of $(T^-)^2$, we find (using $D=e^2+8wh$) that
\[\langle\cdot,\cdot\rangle_\pm|_{\Im T^\pm}\equiv 0\mod 2\iff e\pm1=e+\ep\equiv 0\mod 4,\]
as claimed.
\end{proof}

\begin{rem}
Note that Lanneau and Nguyen use a similar idea (restriction of the intersection pairing to the image of an operator $\mod 2$) to show that there are in fact two distinct components of $\W$ for $D\equiv 1\mod 8$ \cite[Theorem 6.1]{lanneaunguyen}. However, they use a different operator $T=T(P)$ for every prototype and this does not seem a feasible invariant.
\end{rem}

\begin{proof}[Proof of \autoref{spin}]
Let $D$ be an odd discriminant. We denote by $\X\rightarrow\W$ the universal family over the Teichmüller curve $\W$, see \cite[\S 1.4]{VHS}. By definition of $\W$, each fibre $\X_t$ has an involution $\rho_t$ and the real multiplication gives endomorphisms $T_t$ of $\H_1(\X_t,\Z)^-$, allowing us to consider the restriction of the intersection form $\langle\cdot,\cdot\rangle_t$ to the image of $T_t$ and take $\Z/2\Z$ coefficients. In particular, the map 
\[t\mapsto\langle\cdot,\cdot\rangle|_{\Im T_t}\mod 2\]
is continuous and as the range (the space of bilinear operators on an $\mathbb{F}_2$ vector space) is discrete, it is locally constant. Now, \autoref{spininvariant} asserts that two cusp prototypes $P_D$, $P_D'$ are associated to cusps on the same component if and only if $e+\ep\equiv e'+\ep'\mod 4$ and, as any such $e$ must be odd, this yields the claim.
\end{proof}


\section{The Galois Action on the Components}
\label{galoisaction}

The aim of this section is to prove \autoref{comphomeo}. The idea is to show that, for $D\equiv 1\mod 8$, the two components of $\W$ are in fact Galois-conjugate in analogy to the situation in genus $2$ (cf. \cite[Theorem 3.3]{BMFuchs}).

To achieve this, we first describe algebraic models of the stable curves associated to the cusps of $\W$ and then describe the Galois-action on these curves explicitly.

%

\addsubsection{Stable Curves} While a Teichmüller curve $\CC$ is never compact, it admits a smooth completion $\overline{\CC}$. Moreover, after passing to a finite cover, we may pull back the universal family over $\M_g$ to $\CC$, thus obtaining a family of curves, which we -- by abuse of notation -- also denote by $\X\rightarrow\CC$ and which extends to a family of stable curves $\overline{\X}\rightarrow\overline{\CC}$, cf. \cite[\S 1.4]{VHS}. 

Much of the geometry of the stable fibres is given by the flat structure. By the above, given a flat surface $(X,\omega)$ on $\X$ together with a periodic direction $v$, we may associate a cusp $(X_\infty,\omega_\infty)$ to $(X,\omega,v)$, where $X_\infty$ is a stable curve and $\omega_\infty$ is a stable differential on $X_\infty$, see e.g. \cite[\S 2.5 and \S 5.4]{parkcity}. In particular, $X_\infty$ is obtained from~$X$ topologically by contracting the core curves of cylinders and $\omega_\infty$ has poles with residue equal to the cylinder widths at the nodes of $X_\infty$.

\begin{lemma}\label{stablefibres}
Let $c\in\overline{\W}\setminus\W$ be a point such that the fibre $X_\infty=\overline{\X}_c$ is singular. Then $X_\infty$ is a trinodal curve, i.e. $X_\infty$ is a projective line with three pairs of points identified.
\end{lemma}

\begin{proof}
This follows immediately from \cite[Corollary 5.11]{parkcity}: let $(X_\infty,\omega_\infty)$ be the stable flat surface associated to $c$. Then, as every component of $X_\infty$ must contain a zero of $\omega_\infty$, the stable curve $X_\infty$ is irreducible. Moreover, $(X_\infty,\omega_\infty)$ is obtained by contracting the core curves of a cylinder decomposition on some $(X,\omega)\in\W$. But by \autoref{3cylinders}, any such $(X,\omega)$ decomposes into three cylinders, hence $X_\infty$ is obtained topologically by contracting three (homologically independent!) curves on a genus $3$ Riemann surface and therefore has geometric genus $0$ and three nodes. 
\end{proof}

Using the prototypes of \cite{lanneaunguyen} from \autoref{cusps}, we can describe the singular fibres of $\overline{\W}$ more explicitly, in the spirit of \cite[Proposition 3.2]{BMFuchs}.

\begin{prop}\label{stablecoordinates}
The stable curve above the cusp associated to the combinatorial prototype $[w,h,t,e,\ep]$ may 
be normalised by a projective line with six marked points: $\pm1,\pm x_1$, and $\pm x_3$, where
\[x_1=-s-\sqrt{\frac{1-s^2}{3}}\quad\text{and}\quad x_3=-s+\sqrt{\frac{1-s^2}{3}}\quad\text{for}\quad s=\begin{dcases}\frac{e+\sqrt{D}}{4w},&\text{if $\ep=1$,}\\\frac{2w}{e+\sqrt{D}},&\text{if $\ep=-1$,}\end{dcases}\]
and the pairs of points $(+1,-1)$, $(x_1,-x_3)$, and $(x_3,-x_1)$ are identified in the stable model.

In particular, 
the absolute value of $s$ uniquely determines the stable fibres.
\end{prop}

\begin{proof}
By \autoref{stablefibres}, the normalisation of the stable curve $X_\infty$ associated to a cusp of $\W$ is a projective line with three pairs of marked points which we denote by $x_1,y_1,x_2,y_2,x_3,y_3$. 

Now, the stable differential $\omega_\infty$ has poles at the nodes of $X_\infty$ and the residues at each node must add up to zero, i.e. we have the crossratio equation
\begin{equation}\label{crossratio}
\omega_\infty=\biggl(\sum_{i=1}^3\frac{r_i}{z-x_i}-\frac{r_i}{z-y_i}\biggr)\d z=\frac{C \d z}{\displaystyle\prod_{i=1}^3 (z-x_i)(z-y_i)},
\end{equation}
for the residues $r_i$, some constant $C$, and after choosing coordinates so that the unique zero of $\omega_\infty$ is at $\infty$.

Moreover, the Prym involution $\rho$ acts on $X_\infty$, hence also on the normalisation, where we choose coordinates so that it acts as $z\mapsto -z$ (fixing the zero at $\infty$) and $x_2=1$. Recall that the stable fibre was obtained topologically by contracting the core curves of the three cylinders and that two cylinders are exchanged by the involution, one is fixed. We therefore find
\[y_1=-x_3,\quad y_2=-x_2=-1,\quad y_3=-x_1,\quad\text{and}\quad r_1=r_3.\]
Comparing coefficients in \eqref{crossratio}, we obtain
\[x_1=-x_3-2s\quad\text{and}\quad x_3=-s\pm\sqrt{\frac{1-s^2}{3}}\quad\text{for}\quad s=\frac{r_2}{2r_1}.\]
Observe that the choice of sign in $x_3$ interchanges the values of $x_1$ and $x_3$ and that $-s$ gives the same set of points.

Now, consider the cusp associated to the prototype $[w,h,t,e,\ep]$. If $\ep=1$, we have $r_1=r_3=w$ and $r_2=\lambda$, while $\ep=-1$ implies $r_1=r_3=\nicefrac{\lambda}{2}$ and $r_2=w$ (cf. \autoref{fig:prototypes}). This determines $s$.

Conversely, $\abs{s}$ determines the points $x_i$. Identifying the points $\pm 1$, $x_1$ and $-x_3$, and $x_3$ and $-x_1$, we obtain a stable curve with three nodes and an involution.
%
\end{proof}

\begin{rem}\label{ssign}
Note that replacing $s$ with $-s$ in \autoref{stablecoordinates} gives the same six points on $\P^1$, i.e. the same stable curve. This ambiguity corresponds to the action of the Prym involution on the stable curve.
\end{rem}

\begin{rem}
Observe that the stable curve does not ``see'' the twist parameter~$t$, as it only depends on the cylinder widths. In particular, cusp prototypes that differ only in their twist parameter cannot be distinguished by the associated stable curves. This motivates the following definition.
\end{rem}

\begin{dfn}
Given a prototype $P=[w,h,t,e,\ep]$, we define the associated \emph{algebraic cusp prototype} as $[w,h,e,\ep]$.
\end{dfn}



\addsubsection{The Galois Action} As Teichmüller curves are rigid, they are defined over a number field \cites{mcmullenrigid}{moellerviehweg}. In particular, 
the absolute Galois group $\Gal(\overline{\Q}/\Q)$ acts on the set of all Teichmüller curves, hence also on the set of cusps of Teichmüller curves. 

Using the algebraic description of the stable curves, we may describe the Galois action on the cusps of $\W$. As this is again independent of the twist parameter $t$, the action is given only on algebraic cusp prototypes. 

\begin{prop}\label{galoiscusps}
Let $P=[w,h,e,\ep]$ be an algebraic cusp prototype, let $\sigma\in\Gal(\overline{\Q}/\Q)$ 
be a Galois-automorphism that maps $\sqrt{D}$ to $-\sqrt{D}$,
and denote by $P^\sigma$ the prototype corresponding to the $\sigma$-conjugate cusp. Then, if $\ep=1$,
\[P^\sigma=%
\begin{cases}
[h,w,e,-\ep],&\text{if $h>\nicefrac{\lambda}{2}$},\\
[w,h,-e,\ep],&\text{if $h<\nicefrac{\lambda}{2}$},\\
\end{cases}\]
and if $\ep=-1$,
\[P^\sigma=%
\begin{cases}
[h,w,e,-\ep],&\text{if $h>\lambda$},\\
[w,h,-e,\ep],&\text{if $h<\lambda$},\\
\end{cases}\]
where $2\lambda=e+\sqrt{D}$, as above.
\end{prop}

\begin{proof}
Let $P=[w,h,e,\ep]$ be an algebraic cusp prototype. 
By \autoref{stablecoordinates}, the conjugate cusp will depend only on the action of $\sigma$ on $s$. Recall that for $\ep=1$, we have
\[s=s(P)=s^+=\frac{e+\sqrt{D}}{4w},\quad\text{i.e.}\quad (s^+)^\sigma=-\frac{-e+\sqrt{D}}{4w},\]
while for $\ep=-1$
\[s=s(P)=s^-=\frac{2w}{e+\sqrt{D}}=\frac{-e+\sqrt{D}}{4h},\quad\text{i.e.}\quad (s^-)^\sigma=-\frac{e+\sqrt{D}}{4h},\]
as $D=e^2+8wh$.

Now, consider a prototype $P'=[w',h',e',\ep']$ such that $\abs{s(P)^\sigma}=\abs{s(P')}$ (recall that by \autoref{ssign}, $s$ is determined only up to sign, due to the action of the Prym involution). Comparing coefficients in $\Q(\sqrt{D})$ and as $w,h>0$, it is clear that either $\ep'=-\ep$ and $e'=e$ or $e'=-e$ and $\ep'=\ep$. In the first case, $w'=h$ and $h'=w$, while in the second case $w'=w$ and $h'=h$.

Moreover, observe (using again that $D=e^2+8wh$) that
\[h<\frac{\lambda}{2}\iff e+\sqrt{D}>4h=\frac{D-e^2}{2w}
\iff \frac{\sqrt{D}-e}{2}<w,\]
and that any valid prototype $[w',h',e',1]$ must satisfy $w'>\lambda'$.
Hence, comparing $h$ to $\lambda$ (respectively $\nicefrac{\lambda}{2}$), determines which of the above described choices for $P'$ gives a valid prototype and thus yields the claim.
\end{proof}


We now combine \autoref{spin} with \autoref{galoiscusps} to show that, when $D$ is odd any two conjugate cusps are on different components.

\begin{prop}\label{switching}
Let $D\equiv 1\mod 8$ and $P_D=[w,h,e,\ep]$ be an algebraic cusp prototype. Then $P_D$ and $P_D^\sigma$ are on different components of $\W$. 

In particular, the cusps associated to $\bigl[\frac{D-1}{8},1,-1,-1\bigr]$ and $\bigl[\frac{D-1}{8},1,1,-1\bigr]$ lie on $\W^1$ and $\W^2$, respectively, and are conjugate.
\end{prop}

\begin{proof}
Let $P_D=[w,h,e,\ep]$ be an algebraic cusp prototype and denote by
\[c(P)=e+\ep\mod 4,\]
the component (see \autoref{spin}) of $\W$ that the associated cusp(s) of $P$ lie on. Then, by \autoref{galoiscusps}, we have
\[c(P^\sigma)\equiv -e+\ep\equiv e-\ep\mod 4,\]
as both $e$ and $\ep$ are $\pm 1\mod 4$. In particular, $c(P)\not\equiv c(P^\sigma)\mod 4$, hence the cusps lie on alternate components.
\end{proof}

\begin{proof}[Proof of \autoref{comphomeo}:]
Let $D\equiv 1\mod 8$, non-square, and $\W^i$ be a Teichmüller curve. Now, $\Gal(\overline{\Q}/\Q)$ acts on $\W^i$ and as this action extends to an action on the families of curves and their Jacobians, respects the (Prym) splitting, and maps eigenforms for real multiplication to eigenforms (for the same $D$!), it preserves the locus $\W$. Hence, any given element of $\Gal(\overline{\Q}/\Q)$ acts either trivially or interchanges the two components. But by \autoref{switching}, there exists an automorphism that \textit{does not} fix $\W^i$ and therefore the components are Galois-conjugate. In particular, they are homeomorphic.
\end{proof}

\printbibliography

\end{document}